\documentclass[a4paper,11pt]{article}
\usepackage{amsmath}
\usepackage{amsfonts}
\usepackage[latin9]{inputenc}
\usepackage[english]{varioref}
\usepackage{dcolumn}
\usepackage[height=22cm , width = 16cm , top = 4cm , left = 3cm, a4paper]{geometry}
\usepackage[a4paper]{geometry}
\usepackage[final]{graphicx}
\usepackage{epsfig}
\usepackage{pstricks}
\usepackage{psfrag}
\usepackage{rotating}
\usepackage{booktabs}
\usepackage{delarray}
\usepackage{amssymb}
\usepackage{rotating}
\usepackage{subfigure}
\usepackage{layout}
\usepackage{amsthm}
\usepackage{xcolor}
\usepackage{setspace}

{
	{
		{
			{

				\newtheorem{thm}{Theorem}[section]
				
				\newtheorem{lem}[thm]{Lemma}
				\newtheorem{propos}[thm]{Proposition}

				\parindent 0pt
				\newcommand{\enter}{\bigskip}

\begin{document}
\thispagestyle{empty}
\author{Sanjiv Kumar Bariwal${}^1$\footnote{{\it{${}$ Email address:}} p20190043@pilani.bits-pilani.ac.in}, Ankik Kumar Giri${}^2$, Rajesh Kumar${}^1$\\
\footnotesize ${}^1$Department of Mathematics, Birla Institute of Technology and Science, Pilani,\\ \small{ Pilani-333031, Rajasthan, India}\\
\footnotesize ${}^2$Department of Mathematics, Indian Institute of Technology Roorkee,\\ \small{ Roorkee-247667, Uttarakhand, India}\\
}

\title{{ Convergence and error analysis for pure collisional breakage equation}}	
									
\maketitle

\begin{quote}
{\small {\em\bf Abstract}}:	 Collisional breakage in the particulate process has a lot of recent curiosity. We study the pure collisional breakage equation which is nonlinear in nature accompanied by  locally bounded breakage kernel and  collision kernel. The continuous equation is discretized using a finite volume scheme (FVS) and the weak convergence of the approximated solution towards the exact solution is analyzed  for non-uniform mesh. The idea of the analysis is based on the weak $L^1$ compactness and a suitable stable condition on time step is introduced. Furthermore, theoretical error analysis is developed for a uniform mesh when kernels are taken in $W_{loc}^{1,\infty}$ space.  The scheme is shown to be first-order convergent which is  verified numerically for three test  examples of the kernels.
	\end{quote}
\section{Introduction}
Particulate processes are prominent in the dynamics of particle development and describe how particles might unite to generate larger ones or break into smaller ones.
 Suppose that each particle is entirely defined by a single size variable, such as its volume or mass. Particle breakage is categorized into linear breakage and collision breakage.
 The linear breakage equation's success is well-known  in investigating phenomena of importance in various scientific areas ranging from engineering, see \cite{diemer2021applications, danha2015application}  and further citations. It is believed that its expansion is required to broaden the variety of procedures that may be evaluated and to increase analysis quality. One conceivable expansion is to include nonlinearity in the breaking process which can occur when the breakage behaviour of a particle is determined not only by its characteristics and dynamic circumstances (as in linear breakage), but also by the state and properties of the entire system, i.e., by binary interactions, collisional breakage could enable some mass transfer between colliding particles. As a result, daughter particles with more extensive volumes than the parent particles are generated.  Non-linear models emerge in a wide range of contexts, including milling and crushing processes \cite{spampinato2017modelling,chen2020collision}, bulk distribution of asteroids \cite{kudzotsa2013mechanisms,yano2016explosive}, fluidized beds \cite{cruger2016coefficient,lee2017development}, etc. \\
Cheng and Redner \cite{cheng1988scaling} used the following integro-differential equation to derive the collisional breakage equation (CBE). It illustrates the time progession of  particle size distribution  $c(t,x)\geq 0$ of particles of mass $x \in ]0,\infty[$ at time $t\geq 0$ and is defined by
\begin{align}\label{maineq}
\frac{\partial{c(t,x)}}{\partial t}=  \int_0^\infty\int_{x}^{\infty} K(y,z)b(x,y,z)c(t,y)c(t,z)\,dy\,dz -\int_{0}^{\infty}K(x,y)c(t,x)c(t,y)\,dy
\end{align}
with the given initial data 
\begin{align}\label{initial}
c(0,x)\ \ = \ \ c^{in}(x) \geq 0, \ \ \ x \in ]0,\infty[.
\end{align}
The characteristics $t$ and $x$ are regarded as dimensionless quantities without losing any generality. In Eq.(\ref{maineq}), the collision kernel $K(x,y)$ depicts the rate of  collision for breakage event  between two particles of volumes $x$ and $y$. In practice, it is assumed that the  collision rate between the  particles of volumes $x$ and $y$ is identical to the collision between $y$ and $x$. The term  $b(x,y,z)$ is called the breakage distribution function, which defines the rate for production of particles of volume $x$  by breakage of  particle of volume $y$ due to interaction between particles of volumes of $y$ and $z$. Breakage distribution function $b$ holds
$$b(x,y,z)\neq 0\,\,\, \text{for} \hspace{0.4cm} x \in (0,y)\,\,\, \text{and} \hspace{0.4cm} b(x,y,z)=0 \,\,\,\text{for} \hspace{0.4cm} x> y$$
as well as satisfies
$$\int_{0}^{y}xb(x,y,z)\,dx=y$$
for all $(y,z) \in {]0,\infty[}^{2}$.  The first term in Eq.(\ref{maineq}) explains gaining particles of volume $x$ due to collision between particles of volumes $y$ and $z$, known as the birth term. The second term is labeled as the death term and describes the disappearance of particles of volume $x$ due to collision with particles of volume $y$. \\

It is also necessary to specify some integral features of the number density function $c(t,x)$, known as moments. The following equation defines the $j^{th}$ moment of the solution as
\begin{align}\label{momemt}
M_{j}(t)=\int_{0}^{\infty}x^{j}c(t,x)\,dx.
\end{align}
The zeroth and first moments are proportional to the total number of particles in the system and its total volume, respectively. Here, $M_{1}^{in}$ denotes the initial volume of the particles in the closed particulate system.\\

Before venturing into the specifics of the current work, let us review the existing literature on the linear breakage equation that has been widely investigated over the years concerning its analytical solutions \cite{ziff1991new}, similarity solutions \cite{peterson1986similarity,breschi2017note}, numerical results \cite{ hosseininia2006numerical,kumar2008convergence}.

 There is a substantial body of work on the well-posedness of the coagulation and linear breakage equations (CLBE). The authors examined the existence and uniqueness of solutions to CLBE with nonsingular coagulation kernels with different growth parameters on the breakage function in \cite{barik2018note,ghosh2018existence,saha2015singular}. Moreover, many publications are accessible for solving coagulation-fragmentation equations numerically, including the method of moments \cite{attarakih2009solution}, finite element scheme \cite{ahmed2013stabilized}, Monte Carlo methodology \cite{lin2002solution}, and  finite volume method (FVM)\cite{kumar2013moment,bourgade2008convergence,forestier2012finite} to name a few. It has been reported by several authors that the FVM is an appropriate option  among the other  numerical techniques  for solving coagulation-fragmentation equations due to its mass conservation property.\\

Collisional breakage model is discussed in just a few mathematical articles in the literature, see \cite{laurenccot2001discrete,walker2002coalescence, barik2020global}. The authors explained the global classical solutions of coagulation and collisional equation with collision kernel, growing indefinitely for large volumes in \cite{barik2020global}. In addition, fewer publications in the physics literature \cite{cheng1990kinetics,ernst2007nonlinear,krapivsky2003shattering} are committed to the collision breakage equation, with the majority dealing with scaling behavior and shattering transitions.  Lauren\c{c}ot and Wrzosek \cite{laurenccot2001discrete} investigated the existence of a solution for discrete collisional breakage with coagulation equation in which they have used the following constraint over the kernels 
$$K(x,y) \leq (xy)^{\alpha}, \,\, \alpha\in [0,1) \,\, \text{and}\,\, b(x,y,z)\leq P< \infty,\,\, \text{for}\,\, 1\leq x< y.$$
 They have also explored gelation and the long-term behavior of solutions. Further, in \cite{giri2021weak}, the analysis is worked out with the coagulation dominating process for mass conserving solutions when the collision kernel grows at most linearly at infinity. To the best of the authors' knowledge, none of the prior studies account for the weak convergence of the numerical scheme for solving collisional breakage equation. Therefore, this article is an attempt to study the  weak convergence analysis of the model for nonsingular unbounded kernels in a numerical sense and then error estimation for kernels in $W_{loc}^{1,\infty}$ space over a uniform mesh. Thanks to the idea taken from  Bourgade and Filbet \cite{bourgade2008convergence}, in which they have treated coagulation and binary fragmentation equation. The proof is based on the weak $L^1$ compactness method.\\

To proceed further, firstly, we will concentrate on the functional setting as having in mind that expected mass conservation in (\ref{momemt}) is necessary. Besides the first moment, the total number of particles in the system must be finite. Therefore, we construct the solution space that exhibits the convergence of the discretized numerical solution to the weak solution of the collisional breakage equation (\ref{maineq}), which is recognized as a weighted $L^1$ space such as 
	
	\begin{align*}
		X^{+} = \{c\in L^{1}(\mathbb{R}^{+})\cap L^1(\mathbb{R}^{+}, x\,dx): c\geq 0, \|c\|< \infty\},
	\end{align*}
where $\|c\|= \int_{0}^{\infty}(1+x)c(x)\,dx,$ for the non-negative initial condition $c^{in} \in X^{+}$ and $\mathbb{R}^{+}=]0,\infty[.$ Here the notation $L^1({\mathbb{R}}^{+}, xdx)$ stands for the space of the Lebesgue measurable real-valued functions on $\mathbb{R}^{+}$ which are integrable with respect to the measure $x\, dx.$\\
	
Now, the specifications of the collisional kernel $K$ and breakage distribution function $b$ are expressed in the following expression: both functions are symmetric and measurable over the domain.	There exist  $\zeta, \eta$  with $0<\zeta \leq \eta \leq 1, \, \zeta+ \eta \leq 1 $ and $\alpha \in \mathbb{R}$, $\lambda >0$ such that 	
\begin{align}\label{breakage funcn}
H1:  \hspace{0.2cm}	b\in L_{\text{loc}}^\infty{(\mathbb{R}^{+} \times \mathbb{R}^{+} \times \mathbb{R}^{+}) },
\end{align}
H2:  
\begin{equation}\label{Collisional func}
K(x,y)=\left\{
 \begin{array}{ll}
				           
 \lambda xy & \quad (x,y)\in (0,1)\times (0,1)\\
 \lambda xy^{-\alpha} & \quad (x,y)\in (0,1)\times (1,\infty)\\
 \lambda x^{-\alpha}y & \quad (x,y)\in (1,\infty)\times (0,1)\\
\lambda(x^{\zeta}y^{\eta}+x^{\eta}y^{\zeta}) & \quad (x,y)\in (1,\infty)\times (1,\infty). 
 \end{array}
\right.
\end{equation}
This article's contents are organized as follows. The discretization methodology based on the FVM and non-conservative form of fully discretized CBE are introduced in Section \ref{scheme}, followed by the detailed convergence analysis in Section \ref{convergence}. In Section \ref{Error}, we examine the first order error estimates of FVM on uniform meshes. Additionally, we have justified  the theoretical error estimation via numerical results in section \ref{testing}. Consequently, in the final Section, some  conclusions are presented.
\section{Numerical Scheme}\label{scheme}
In this section, we commence exploring the FVM for the solution of Eq.(\ref{maineq}). It is based on the spatial domain being divided into tiny grid cells.
 Particle volumes ranging from $0$ to $\infty$ are taken into account in Eq.(\ref{maineq}). Nevertheless, we define the particle volumes to be in a finite domain for practical purposes. Consider the reduced computational domain for volumes (0,$R$], with $0<R <\infty$. Thus the collisional breakage equation is  truncated as 
		\begin{align}\label{trunceq}
				\frac{\partial{c(t,x)}}{\partial t}=  \int_0^R\int_{x}^{R} K(y,z)b(x,y,z)c(t,y)c(t,z)\,dy\,dz -\int_{0}^{R}K(x,y)c(t,x)c(t,y)\,dy
					\end{align}
					with the given initial distribution 
					\begin{align}\label{initial}
								c(0,x)\ \ = \ \ c^{in}(x) \geq 0, \ \ \ x \in ]0,R].
					\end{align}	
					
Consider a partitioning  of the operating domain $(0,R]$ into small cells as $\Lambda_i^h:=]x_{i-1/2}, x_{i+1/2}], \,\,i=1,2, . . . , \mathrm{I}$,	where,\,\, $x_{1/2}=0, \ \ x_{\mathrm{I}+1/2}= R, \hspace{0.2cm} \Delta x_i=x_{i+1/2}-x_{i-1/2}$ 
and  consider	$	h= max \, \Delta x_i \  \forall \ i. $ 	The grid points are the midpoints of each subinterval and are designated as $$x_i=(x_{i-1/2}+x_{i+1/2})/2 \hspace{0.3cm} \text{for} \,\,  i=1,2, . . . , \mathrm{I}.$$
Now, the expression of the mean value of the number density function $c_i(t)$ in the cell $\Lambda_i^h$ is determined  by 	
		\begin{align}\label{meandens}
							c_{i}(t)=\frac{1}{\Delta x_i}\int_{x_{i-1/2}}^{x_{i+1/2}}c(t,x)\,dx,
						\end{align}
where   $\Delta x_i=x_{i+1/2}-x_{i-1/2}$ for $i=1,2, . . . , \mathrm{I}.$	The domain is confined in the range [0, T] for the time parameter, and it is discretized into $N$ time intervals with time step $\Delta t$. The interval is defined as 		
$$ \tau_n=[t_n,t_{n+1}[ \hspace{0.2cm} \text{with}\,\, t_n=n\Delta t,\  n=0,1,...,N-1.$$ 
We now begin developing the scheme on non-uniform meshes. It has the significant advantage of allowing the inclusion of a more extensive domain with fewer mesh points than a uniform mesh. The discretization differs slightly from that of Filbet and Laurencot \cite{bourgade2008convergence},  where they first converted the model (\ref{maineq}) to a conservative equation using Leibniz integral rule, then discretized using FVM.  Although, in this work, we have developed a non-conservative scheme using FVM from the continuous equation (\ref{maineq}).

 To derive the discretized version of the CBE (\ref{trunceq}), we  proceed as follows: integrate the Eq.(\ref{trunceq}) with respect to $x$ over  $i^{th}$ cell yields the following discrete form
\begin{align}\label{semi}
\frac{dc_i}{dt}=B_{C}(i)-D_{C}(i),
\end{align}
where
\begin{align*}
B_{C}(i)=\frac{1}{\Delta x_i}\int_{x_{i-1/2}}^{x_{i+1/2}}\int_0^{x_{\mathrm{I}+1/2}}\int_{x}^{x_{\mathrm{I}+1/2}} K(y,z)b(x,y,z)c(t,y)c(t,z)dy\,dz\,dx
\end{align*}
\begin{align*}
D_{C}(i)= \frac{1}{\Delta x_i}\int_{x_{i-1/2}}^{x_{i+1/2}}\int_0^{x_{\mathrm{I}+1/2}} K(x,y)c(t,x)c(t,y)dy\,dx
\end{align*}
along with initial distribution,
\begin{align}
c_{i}(0)=c_{i}^{in}=\frac{1}{\Delta x_i}\int_{x_{i-1/2}}^{x_{i+1/2}}c_{0}(x)\,dx.
\end{align}	
Implementing the midpoint rule to all of the above representation yields the semi-discrete equation after some simplifications as

\begin{align}\label{semidiscrete}
\frac{dc_{i}}{dt}=&\frac{1}{\Delta x_i}\sum_{l=1}^{\mathrm{I}}\sum_{j=i}^{\mathrm{I}}K_{j,l}c_{j}(t)c_{l}(t)\Delta x_{j}\Delta x_{l}\int_{x_{i-1/2}}^{p_{j}^{i}}b(x,x_{j},x_{l})\,dx-\sum_{j=1}^{\mathrm{I}}K_{i,j}c_{i}(t)c_{j}(t)\Delta x_{j}, 
\end{align}
where the term $p_{j}^{i}$ is expressed by
\begin{equation}
p_{j}^{i} =
\begin{cases}
x_{i}, & \text{if }\,j=i \\
x_{i+1/2}, & j\neq i.							
\end{cases}
\end{equation}
Now, to obtain a fully discrete system, applying explicit Euler discretization to time variable $t$ leads to
\begin{align}\label{fully}
c_{i}^{n+1}-c_{i}^{n}=&\frac{\Delta t}{\Delta x_{i}}\sum_{l=1}^{\mathrm{I}}\sum_{j=i}^{\mathrm{I}}K_{j,l}c_{j}^{n}c_{l}^{n}\Delta x_{j}\Delta x_{l}\int_{x_{i-1/2}}^{p_{j}^{i}}b(x,x_{j},x_{l})\,dx  \nonumber\\
&-\Delta t \sum_{j=1}^{\mathrm{I}}K_{i,j}c_{i}^{n}c_{j}^{n}\Delta x_{j}.
\end{align}
For the convergence analysis, consider a function $c^h$ on $[0,T]\times ]0,R]$ which is representated  by
\begin{align}\label{chap2:function_ch}
c^h(t,x)=\sum_{n=0}^{N-1}\sum_{i=1}^{\mathrm{I}}c_i^n\,\chi_{\Lambda_i^h}(x)\,\chi_{\tau_n}(t),
\end{align}
where $\chi_D(x)$ denotes the characteristic function on a set $D$ as $\chi_D(x)=1$ if $x\in D$ or $0$ everywhere else. Also noting that $$c^h(0,\cdot)=\sum_{i=1}^{\mathrm{I}}c_i^{in} \chi_{\Lambda_i^h}(\cdot)$$
converges strongly to $c^{in}$ in $L^1((0,R))$ as $h\rightarrow 0$. A finite volume approximation approaches the kernels on each space cell, i.e., for all $(u,v)\in ]0,R]\times ]0,R]$ and $(u,v,w)\in ]0,R]\times ]0,R]\times ]0,R],$
\begin{align}\label{chap2:function_aggregatediscrete}
	K^h(u,v)= \sum_{i=1}^{\mathrm{I}} \sum_{j=1}^{\mathrm{I}} K_{i,j} \chi_{\Lambda_i^h}(u) \chi_{\Lambda_j^h}(v),
\end{align}
\begin{align}\label{chap2:function_brkdiscrete}
	b^h(u,v,w)= \sum_{i=1}^{\mathrm{I}} \sum_{j=1}^{\mathrm{I}}\sum_{l=1}^{\mathrm{I}} b_{i,j,l} \chi_{\Lambda_i^h}(u) \chi_{\Lambda_j^h}(v)\chi_{\Lambda_l^h}(w),
	\end{align}

where $$K_{i,j}= \frac{1}{\Delta x_i \Delta x_j} \int_{\Lambda_j^h} \int_{\Lambda_i^h} K(u,v)du\,dv, \quad  b_{i,j,l}= \frac{1}{\Delta x_i \Delta x_j \Delta x_l}\int_{\Lambda_l^h} \int_{\Lambda_j^h} \int_{\Lambda_i^h} b(u,v,w)du\,dv\,dw.$$
Such discretization ensures that {$\|K^h-K\|_{L^1((0,R)\times (0,R))}\rightarrow 0$}, {$\|b^h-b\|_{L^1((0,R)\times (0,R)\times (0,R))}\rightarrow 0$} as $h\rightarrow 0$, see \cite{bourgade2008convergence}.

\section{Weak Convergence}\label{convergence}
The objective of this section is to study the convergence of solution $c^h$ to a function $c$  as $h$ and $\Delta t$ $\rightarrow 0$. 

\begin{thm}\label{maintheorem}
	Consider that $c^{in}\in X^+$ and the hypothesis $(H1)-(H2)$ on kernels hold. Also assuming that under the time step $\Delta t$ and for a constant $\theta> 0$, the following stability condition
						\begin{align}\label{22}
							C(R, T)\Delta t\le \theta< 1,
						\end{align}
						holds for 
						\begin{align}{\label{23}}
							C(R, T):= \lambda(2R  \|c^{in}\|_{L^1}\,e^{2\lambda R \|b\|_{L^{\infty}} M_{1}^{in} T} + M_{1}^{in}).
						\end{align}
					Then there exists the extraction of a sub-sequence as $$c^h\rightarrow c\ \
						\text{in}\ \ L^\infty((0,T;L^1\,(0,R)),$$
						for c being the weak solution to (\ref{maineq}) on $[0,T]$ with initial
						datum $c^{in}$. This implies that, the function $c\geq 0$ satisfies\\
						\begin{equation}
							\begin{aligned}\label{convergence0}
								{}	&	\int_0^T\int_0^{R} c(t,x)\frac{\partial\varphi}{\partial
									t}(t,x)dx\,dt +\int_0^{R} c^{in}(x)\varphi(0,x)dx \\
								&-\int_0^T\int_0^{R} \int_0^{R} \int_{x}^{R}\varphi(t,x)K(y,z)b(x,y,z)c(t,y)c(t,z)dy\,dz\,dx\,dt\\
								&+ \int_0^T \int_0^{R} \int_{0}^{R}\varphi(t,x)K(x,y)c(t,x)c(t,y)dy\,dx\,dt
							=0,
							\end{aligned}
						\end{equation}
						for all smooth functions $\varphi$ having compact support in $[0,T]\times ]0,{R}].$
					\end{thm}
	Following the preceding theorem, it is evident that the main motivation here is to demonstrate the weak convergence of the family of functions $(c^h)$ to a function $c$ in $L^1(0,R)$ as $h$ and $\Delta t$ approach zero. The idea is based on the Dunford-Pettis theorem, which establishes a necessary and sufficient condition for $L^1$ compactness in the presence of weak convergence. 

	\begin{thm}\label{maintheorem1}
	Let us take : $|\Omega|<\infty$ and $c^h:\Omega\mapsto \mathbb{R} $ be a sequence in $L^1(\Omega).$
	Assume that the sequence $\{c^h\}$ satisfies
	\begin{itemize}
	\item $\{c^h\}$ is equibounded in $L^1(\Omega)$, i.e.\
	\begin{align}\label{equiboundedness}
	\sup \|c^h\|_{L^1(\Omega)}< \infty
	\end{align}
	\item $\{c^h\}$ is equiintegrable, iff
	\begin{align}\label{equiintegrable}
	\int_\Omega \Phi(|c^h|)dx< \infty
	\end{align}
	for  $\Phi$ being some increasing function taken as $\Phi:[0,\infty[\mapsto [0,\infty[$ such that
	\begin{align*}
	\lim_{r\rightarrow \infty}\frac{\Phi(r)}{r}\rightarrow
								\infty.
		\end{align*}
			\end{itemize}
	Then $c^h$ belongs to a weakly compact set in $L^1(\Omega)$ implying that there is a  subsequence of $c^h$ that weakly converges in $L^1(\Omega)$.
	\end{thm}
	As a result, demonstrating the equiboundedness and equiintegrability of the family $c^h$ in $L^1$ as in (\ref{equiboundedness}) and (\ref{equiintegrable}), respectively, is sufficient to establish Theorem \ref{maintheorem}. The following proposition addresses the  non-negativity and equiboundedness of the functions $c^h$. For the proof, we employed Bourgade and Filbet's approach \cite{bourgade2008convergence}.
	
\begin{propos}
{Assume that the stability criterion (\ref{22}) holds for time step $\Delta t$.
Furthermore, assuming that the kernel growth condition satisfies $(H1)-(H2)$. Then $c^h$ is a non-negative function that fulfills the estimation given below
\begin{align}\label{36}
\int_0^R c^h(t,x)dx \leq  \|c^{in}\|_{L^1}\,e^{2\lambda R \|b\|_{L^{\infty}} M_{1}^{in} t}.
\end{align}
}
\end{propos}
\begin{proof}
Mathematical induction is used to demonstrate the non-negativity and equiboundedness of the function $c^h$. At t = 0, it is known that $c^h(0) \geq 0$ and belongs to $L^1(0,R)$. Assuming that the functions $c^h(t^n)\geq 0 $ and
 	\begin{align}\label{mon1}
 	\int_0^{R} c^h(t^n,x)dx \le  \|c^{in}\|_{L^1}\,e^{2\lambda R \|b\|_{L^{\infty}} M_{1}^{in} t^n}.
 	\end{align}
 	Then, our first goal is to demonstrate that $c^h(t^{n+1})\geq 0$. Consider the cell at the boundary with index i = 1. As a result, in this situation, we obtain from Eq.(\ref{fully}),
 	\begin{align}\label{non}
 	c_{1}^{n+1}=& c_{1}^{n}+\frac{\Delta t}{\Delta x_{1}}\sum_{l=1}^{\mathrm{I}}\sum_{j=1}^{\mathrm{I}}K_{j,l}c_{j}^{n}c_{l}^{n}\Delta x_{j}\Delta x_{l}\int_{x_{1/2}}^{p_{j}^{1}}b(x,x_{j},x_{l})\,dx  \nonumber\\
 	&-\Delta t \sum_{j=1}^{\mathrm{I}}K_{1,j}c_{1}^{n}c_{j}^{n}\Delta x_{j} \nonumber \\
 	&\geq c_{1}^{n}-\Delta t \sum_{j=1}^{\mathrm{I}}K_{1,j}c_{1}^{n}c_{j}^{n}\Delta x_{j}.
 	\end{align}
	Moving further, we choose the first case for collisional kernel, case-(1):
	 $ K(x,y)= \lambda(x^{\zeta}y^{\eta}+x^{\eta}y^{\zeta}),\, \text{when}  \,\, (x,y)\in (1,R)\times (1,R)$ 
	 \begin{align*}
	 c_{1}^{n+1}\geq  c_{1}^{n}-\Delta t \sum_{j=1}^{\mathrm{I}}\lambda(x_{1}^{\zeta}x_{j}^{\eta}+x_{1}^{\eta}x_{j}^{\zeta})c_{1}^{n}c_{j}^{n}\Delta x_{j},
	 \end{align*}
	 using the fact that $\lambda(x_{i}^{\zeta}x_{j}^{\eta}+x_{i}^{\eta}x_{j}^{\zeta})\leq \lambda(x_{i}+x_{j})$, thanks to Young's inequality, convert the above inequality into following one
	 \begin{align}\label{non1}
	 c_{1}^{n+1}&\geq  \,c_{1}^{n}-\lambda\Delta t \sum_{j=1}^{\mathrm{I}}(x_{1}+x_{j})c_{1}^{n}c_{j}^{n}\Delta x_{j} \nonumber \\
	 &\geq [1-\lambda\Delta t(R\sum_{j=1}^{\mathrm{I}}c_{j}^{n}\Delta x_{j}+ M_{1}^{in})]c_{1}^{n}.
	  \end{align} 	
	  Now, consider case-(2): $K(x,y)=\lambda x^{-\alpha}y, \, \text{when}\,\, (x,y)\in (1,R)\times (0,1)$ and put this value in Eq.(\ref{non}), then  imposing the condition $x^{-\alpha}\leq 1$  yields
	  \begin{align}
	   c_{1}^{n+1}&\geq  \,c_{1}^{n}-{\lambda}\Delta t \sum_{j=1}^{\mathrm{I}}x_{j}c_{1}^{n}c_{j}^{n}\Delta x_{j}\nonumber  \\ 
	   & \geq (1-\lambda\Delta t M_{1}^{in})c_{1}^{n}.
	  \end{align}
For	case-(3): $K(x,y)=\lambda xy^{-\alpha}, \, \text{when}\,\, (x,y)\in (0,1)\times (1,R)$ with  $y^{-\alpha}\leq 1$ and for case-(4): $K(x,y)=\lambda (xy), \, \text{when}\,\, (x,y)\in (0,1)\times (0,1)$ provide
\begin{align}
c_{1}^{n+1}\geq (1-\lambda\Delta t \sum_{j=1}^{\mathrm{I}}c_{j}^{n}\Delta x_{j})c_{1}^{n}.
\end{align}

All the results from case(1)-case(4) are collected and the following inequality is achieved
\begin{align}
c_{1}^{n+1}\geq [1-\lambda \Delta t(R\sum_{j=1}^{\mathrm{I}}c_{j}^{n}\Delta x_{j}+M_{1}^{in})]c_{1}^{n}.  
\end{align}
Using conditions (\ref{22}), (\ref{23}) and Eq.(\ref{mon1}), the non-negativity of $c_{1}^{n+1}$ is obtained. Thus, we assume that the computations for $i\geq 2$ goes similar to $i=1$ for all four cases and obtain the results like the previous ones. As a result, applying the stability condition  on the time step $\Delta t$ and the $L^1$ bound on $c^h$  yield $c^h({t^{n+1}})\geq 0.$\\

Following that, it is demonstrated that $c^h(t^{n+1})$ follows a similar estimation as (\ref{mon1}). To see this, multiply equation (\ref{fully}) by the term $\Delta x_i$, leaving the negative term out, and determine the result using summation with respect to $i$, as 
\begin{align}\label{equi1}
\sum_{i=1}^{\mathrm{I}}c_{i}^{n+1}\Delta x_{i} &\leq  \sum_{i=1}^{\mathrm{I}}c_{i}^{n}\Delta x_{i}+\Delta t\sum_{i=1}^{\mathrm{I}}\sum_{l=1}^{\mathrm{I}}\sum_{j=i}^{\mathrm{I}}K_{j,l}c_{j}^{n}c_{l}^{n}\Delta x_{j}\Delta x_{l}\int_{x_{i-1/2}}^{p_{j}^{i}}b(x,x_{j},x_{l})\,dx \nonumber \\
& \leq \sum_{i=1}^{\mathrm{I}}c_{i}^{n}\Delta x_{i}+
\Delta t\|b\|_{\infty}\sum_{i=1}^{\mathrm{I}}\sum_{l=1}^{\mathrm{I}}\sum_{j=1}^{\mathrm{I}}K_{j,l}c_{j}^{n}c_{l}^{n}\Delta x_{j}\Delta x_{l}\int_{x_{i-1/2}}^{x_{i+1/2}}\,dx \nonumber \\
& \leq \sum_{i=1}^{\mathrm{I}}c_{i}^{n}\Delta x_{i}+
\Delta t R \|b\|_{\infty}\sum_{l=1}^{\mathrm{I}}\sum_{j=1}^{\mathrm{I}}K_{j,l}c_{j}^{n}c_{l}^{n}\Delta x_{j}\Delta x_{l}.
\end{align}
Again, the above will be simplified for four cases of kernels:\\
Case-(1): $ K(x,y)= \lambda(x^{\zeta}y^{\eta}+x^{\eta}y^{\zeta}),\, \text{when}  \,\, (x,y)\in (1,R)\times (1,R)$. Substitute the value of $K(x,y)$ in Eq.(\ref{equi1}) and  using the Young's inequality leads to 
\begin{align*}
\sum_{i=1}^{\mathrm{I}}c_{i}^{n+1}\Delta x_{i} &\leq  \sum_{i=1}^{\mathrm{I}}c_{i}^{n}\Delta x_{i} + {\lambda}\Delta t R \|b\|_{\infty}\sum_{l=1}^{\mathrm{I}}\sum_{j=1}^{\mathrm{I}} (x_{j}+x_{l})c_{j}^{n}c_{l}^{n}\Delta x_{j}\Delta x_{l}\nonumber \\
& \leq (1+2{\lambda}\Delta t R \|b\|_{\infty}M_{1}^{in})\sum_{i=1}^{\mathrm{I}}c_{i}^{n}\Delta x_{i}.
\end{align*}
Finally, having (\ref{mon1}) the  $L^1$ bound of $c^h$ at time step n and  $1+x < \exp(x)$ $\forall$ $x>0$ imply that
$$\sum_{i=1}^{\mathrm{I}}c_{i}^{n+1}\Delta x_{i} \leq \|c^{in}\|_{L^1}\,e^{2\lambda R \|b\|_{L^{\infty}} M_{1}^{in} t^{n+1}}.$$
As a consequence, the result (\ref{36}) accomplished.\\
Case-(2): $K(x,y)=\lambda x^{-\alpha}y, \, \text{when}\,\, (x,y)\in (1,R)\times (0,1)$, and case-(3): $K(x,y)=\lambda xy^{-\alpha}, \, \text{when}\,\, (x,y)\in (0,1)\times (1,R)$ have  similar computations. The values of $K(x,y)$ after  substituting in Eq.(\ref{equi1}) yields
\begin{align*}
\sum_{i=1}^{\mathrm{I}}c_{i}^{n+1}\Delta x_{i} &\leq  \sum_{i=1}^{\mathrm{I}}c_{i}^{n}\Delta x_{i} + {\lambda}\Delta t R \|b\|_{\infty}\sum_{l=1}^{\mathrm{I}}\sum_{j=1}^{\mathrm{I}} (x_{j}^{-\alpha}x_{l})c_{j}^{n}c_{l}^{n}\Delta x_{j}\Delta x_{l}\\
&\leq \sum_{i=1}^{\mathrm{I}}c_{i}^{n}\Delta x_{i} + {\lambda}\Delta t R \|b\|_{\infty}\sum_{l=1}^{\mathrm{I}}\sum_{j=1}^{\mathrm{I}} x_{l}c_{j}^{n}c_{l}^{n}\Delta x_{j}\Delta x_{l}\\
& \leq (1+\lambda\Delta t R \|b\|_{\infty}M_{1}^{in})\sum_{i=1}^{\mathrm{I}}c_{i}^{n}\Delta x_{i}.
\end{align*}
Again, using (\ref{mon1}) and $1+x < \exp(x)$ $\forall$ $x>0$  provide the $L^1$ bound for $c^h$ at time step $n+1$.

Case-(4): For  $K(x,y)=\lambda (xy), \, \text{when}\,\, (x,y)\in (0,1)\times (0,1)$,  inserting the value of $K$ in Eq.(\ref{equi1}) employs
\begin{align*}
\sum_{i=1}^{\mathrm{I}}c_{i}^{n+1}\Delta x_{i} &\leq  \sum_{i=1}^{\mathrm{I}}c_{i}^{n}\Delta x_{i} + {\lambda}\Delta t R \|b\|_{\infty}\sum_{l=1}^{\mathrm{I}}\sum_{j=1}^{\mathrm{I}} x_{j}x_{l}c_{j}^{n}c_{l}^{n}\Delta x_{j}\Delta x_{l}\\
& \leq \sum_{i=1}^{\mathrm{I}}c_{i}^{n}\Delta x_{i} + {\lambda}\Delta t R \|b\|_{\infty}\sum_{l=1}^{\mathrm{I}}\sum_{j=1}^{\mathrm{I}} x_{l}c_{j}^{n}c_{l}^{n}\Delta x_{j}\Delta x_{l}.
\end{align*}
 To get the result (\ref{36}) for $c^h(t^{n+1}),$ the computations are similar to the previous case.
\end{proof}

To demonstrate the family of solutions's uniform integrability, let us designate a specific category of convex functions as $C_{V P,\infty}$. Consider $ \Phi \in C^{\infty}([0,\infty))$, a non-negative and convex function that belongs to the ${C}_{VP, \infty}$ class and has the following properties:

	\begin{description}
	\item[(i)] $\Phi(0)=0,\ \Phi'(0)=1$ and $\Phi'$ is concave;
	\item[(ii)] $\lim_{p \to \infty} \Phi'(p) =\lim_{p \to \infty} \frac{ \Phi(p)}{p}=\infty$;
	\item[(iii)] for $\theta \in (1, 2)$,
	\begin{align}\label{Tproperty}
	\Pi_{\theta}(\Phi):= \sup_{p \ge 0} \bigg\{   \frac{ \Phi(p)}{p^{\theta}} \bigg\} < \infty.
	\end{align}
	\end{description}
	It is given that, $c^{in}\in L^1\,(0,R)$, therefore, by De la Vall\'{e}e Poussin theorem, a convex function
	$\mathrm{\Phi}\geq 0$ exists which is  continuously differentiable on
	$\mathbb{R}^{+}$ with $\mathrm{\Phi}(0)=0$, $\mathrm{\Phi}'(0)=1$
	such that $\mathrm{\Phi}'$ is concave
	$$\frac{\mathrm{\Phi}(p)}{p} \rightarrow \infty,\ \ \text{as}\ \
	p \rightarrow \infty$$ and
	\begin{align}\label{convex}
	\mathcal{I}:=\int_0^{R} \mathrm{\Phi}(c^{in})(x)dx< \infty.
	\end{align}
	\begin{lem} [\cite{laurenccot2002continuous}, Lemma B.1.]\label{lemma}
	Let $\mathrm{\Phi}\in {C}_{VP, \infty}$. 
Then $\forall$  $(x,y)\in \mathbb{R}^{+}\times \mathbb{R}^{+},$
	$$x\mathrm{\Phi}'(y)\leq \mathrm{\Phi}(x)+\mathrm{\Phi}(y).$$
	\end{lem}
	
	The equiintegrability is now examined in the following statement.
	\begin{propos}\label{equiintegrability}
							Let $c^{in}\geq 0\in L^1 (0,R)$ and (\ref{fully}) constructs the family $(c^h)$ for any $h$ and $\Delta t$, where $\Delta t$ fulfills the relation (\ref{22}). Then $(c^h)$ is weakly relatively sequentially compact in $L^1((0,T)\times (0,R))$.
						\end{propos}
						
	\begin{proof}
	The objective here is to get a result comparable to (\ref{convex}) for the function family $c^h$.  Using the sequence $c^{n}_{i}$, the integral of $(c^h)$ may be expressed as
		\begin{align*}
		\int_0^T \int_0^{R} \mathrm{\Phi}(c^h(t,x))dx\,dt=&\sum_{n=0}^{N-1}\sum_{i=1}^{\mathrm{I}}\int_{\tau_n}\int_{\Lambda_i^h}\mathrm{\Phi}
		\bigg(\sum_{k=0}^{N-1}\sum_{j=1}^{\mathrm{I}}c_j^k\chi_{\Lambda_j^h}(x)\chi_{\tau_k}(t)\bigg)dx\,dt\\
		=&\sum_{n=0}^{N-1}\sum_{i=1}^{\mathrm{I}}\Delta t\Delta x_i\mathrm{\Phi}(c_i^n).
		\end{align*}
	It follows from the discrete Eq.(\ref{fully}), as well as the convexity of the function $\mathrm{\Phi}$  and  ${\mathrm{\Phi}}^{'}\geq 0$, that
	\begin{align}\label{Equi1}
\sum_{i=1}^{\mathrm{I}} [\mathrm{\Phi}(c_i^{n+1})-\mathrm{\Phi}(c_i^{n})]\Delta x_{i}& \leq \sum_{i=1}^{\mathrm{I}}\left(c_i^{n+1}-c_i^{n}\right)\mathrm{\Phi}^{'}(c_i^{n+1})\Delta x_{i} \nonumber\\
& \leq \Delta t\sum_{i=1}^{\mathrm{I}}\sum_{l=1}^{\mathrm{I}}\sum_{j=i}^{\mathrm{I}}K_{j,l}c_{j}^{n}c_{l}^{n}\mathrm{\Phi}^{'}(c_i^{n+1})\Delta x_{j}\Delta x_{l}\int_{x_{i-1/2}}^{x_{i+1/2}}b(x,x_{j},x_{l})\,dx \nonumber \\
&  \leq \Delta t\sum_{i=1}^{\mathrm{I}}\sum_{l=1}^{\mathrm{I}}\sum_{j=1}^{\mathrm{I}}K_{j,l}c_{j}^{n}c_{l}^{n}\mathrm{\Phi}^{'}(c_i^{n+1})\Delta x_{j}\Delta x_{l} b(x_{i},x_{j},x_{l})\Delta x_{i}.
	\end{align}
Case-(1): $ K(x,y)= \lambda(x^{\zeta}y^{\eta}+x^{\eta}y^{\zeta}),\, \text{when}  \,\, (x,y)\in (1,R)\times (1,R)$. Substitute the value of $K(x,y)$ in Eq.(\ref{Equi1}) yields
\begin{align*}
\sum_{i=1}^{\mathrm{I}} [\mathrm{\Phi}(c_i^{n+1})-\mathrm{\Phi}(c_i^{n})]\Delta x_{i} \leq &  {\lambda} \Delta t\sum_{i=1}^{\mathrm{I}}\sum_{l=1}^{\mathrm{I}}\sum_{j=1}^{\mathrm{I}}(x_{j}+x_{l})c_{j}^{n}\Delta x_{j}c_{l}^{n}\Delta x_{l}\Delta x_{i} b(x_{i},x_{j},x_{l})\mathrm{\Phi}^{'}(c_i^{n+1}).
\end{align*}
The convexity result in Lemma {\ref{lemma}} allows us to obtain
\begin{align}\label{Equi2}
\sum_{i=1}^{\mathrm{I}} [\mathrm{\Phi}(c_i^{n+1})-\mathrm{\Phi}(c_i^{n})]\Delta x_{i} & \leq  2 {\lambda} \Delta t\sum_{i=1}^{\mathrm{I}}\sum_{l=1}^{\mathrm{I}}\sum_{j=1}^{\mathrm{I}}x_{j}c_{j}^{n}\Delta x_{j}c_{l}^{n}\Delta x_{l}\Delta x_{i}[ \mathrm{\Phi}(c_i^{n+1})+{\mathrm{\Phi}}(b(x_{i},x_{j},x_{l}))] \nonumber \\
& \leq 2 {\lambda} \Delta t\sum_{i=1}^{\mathrm{I}}\sum_{l=1}^{\mathrm{I}}\sum_{j=1}^{\mathrm{I}}x_{j}c_{j}^{n}\Delta x_{j}c_{l}^{n}\Delta x_{l}\Delta x_{i}\mathrm{\Phi}(c_i^{n+1})\nonumber\\
& +  2 {\lambda} \Delta t\sum_{i=1}^{\mathrm{I}}\sum_{l=1}^{\mathrm{I}}\sum_{j=1}^{\mathrm{I}}x_{j}c_{j}^{n}\Delta x_{j}c_{l}^{n}\Delta x_{l}\Delta x_{i}{\mathrm{\Phi}}(b(x_{i},x_{j},x_{l})).
\end{align}
After employing the Eq.(\ref{Tproperty}) and Eq.(\ref{36}) into the second term on right-hand side of the above equation leads to
\begin{align}\label{Equi3}
 2 {\lambda} \Delta t\sum_{i=1}^{\mathrm{I}}\sum_{l=1}^{\mathrm{I}}\sum_{j=1}^{\mathrm{I}}x_{j}c_{j}^{n}\Delta x_{j}c_{l}^{n}\Delta x_{l}\Delta x_{i}{\mathrm{\Phi}}(b(x_{i},x_{j},x_{l})) \nonumber\\
 =  2 {\lambda} \Delta t\sum_{i=1}^{\mathrm{I}}\sum_{l=1}^{\mathrm{I}}\sum_{j=1}^{\mathrm{I}}x_{j}c_{j}^{n}\Delta x_{j}c_{l}^{n}&\Delta x_{l}\Delta x_{i}\frac{{\mathrm{\Phi}}(b(x_{i},x_{j},x_{l}))}{\{b(x_{i},x_{j},x_{l})\}^{\theta}}{b(x_{i},x_{j},x_{l})}^{\theta}\nonumber\\
 \leq 2{\lambda} \Delta t R \Pi_{\theta}(\Phi)M_{1}^{in}{\|b\|}^{\theta}_{\infty}\sum_{l=1}^{\mathrm{I}}&c_{l}^{n}\Delta x_{l}\nonumber \\
       \leq 2\lambda  \Delta t R \Pi_{\theta}(\Phi)M_{1}^{in}{\|b\|}^{\theta}_{\infty} &\|c^{in}\|_{L^1}\,e^{2\lambda R \|b\|_{L^{\infty}} M_{1}^{in} T}.
\end{align}
Now, Eq.(\ref{Equi2}) and Eq.(\ref{Equi3}) imply that
\begin{align*}
\sum_{i=1}^{\mathrm{I}} [\mathrm{\Phi}(c_i^{n+1})-\mathrm{\Phi}(c_i^{n})]\Delta x_{i} \leq & 2\lambda \Delta t M_{1}^{in}\|c^{in}\|_{L^1}\,e^{2\lambda R \|b\|_{L^{\infty}} M_{1}^{in} T}\sum_{i=1}^{\mathrm{I}}\Delta x_{i}\mathrm{\Phi}(c_i^{n+1})\nonumber \\
& + 2\lambda  \Delta t R \Pi_{\theta}(\Phi)M_{1}^{in}{\|b\|}^{\theta}_{\infty} \|c^{in}\|_{L^1}\,e^{2\lambda R \|b\|_{L^{\infty}} M_{1}^{in} T}.
\end{align*}
It can be easily simplified as
\begin{align*}
(1-2\lambda \Delta t M_{1}^{in}\|c^{in}\|_{L^1}\,e^{2\lambda R \|b\|_{L^{\infty}} M_{1}^{in} T})\sum_{i=1}^{\mathrm{I}}\Delta x_{i}\mathrm{\Phi}(c_i^{n+1}) \leq & \sum_{i=1}^{\mathrm{I}}\Delta x_{i}\mathrm{\Phi}(c_i^{n}) \nonumber \\
& + 2\lambda \Delta t R \Pi_{\theta}(\Phi)M_{1}^{in}{\|b\|}^{\theta}_{\infty} \|c^{in}\|_{L^1}\,e^{2\lambda R \|b\|_{L^{\infty}} M_{1}^{in} T}. 
\end{align*}
 The above inequality implies that 
\begin{align*}
\sum_{i=1}^{\mathrm{I}}   \Delta x_i \mathrm{\Phi}(c_i^{n+1}) \le A \sum_{i=1}^{\mathrm{I}}   \Delta x_i \mathrm{\Phi}(c_i^{n})+  B,
\end{align*} 
where 
$$  A= \frac{1}{(1- 2 \lambda \Delta t M_{1}^{in}\|c^{in}\|_{L^1}\,e^{2\lambda R \|b\|_{L^{\infty}} M_{1}^{in} T})},  \,\, 
B = \frac{2\lambda \Delta t R \Pi_{\theta}(\Phi)M_{1}^{in}{\|b\|}^{\theta}_{\infty} \|c^{in}\|_{L^1}\,e^{2\lambda R \|b\|_{L^{\infty}} M_{1}^{in} T}}{(1-2\lambda \Delta t M_{1}^{in}\|c^{in}\|_{L^1}\,e^{2\lambda R \|b\|_{L^{\infty}} M_{1}^{in} T})}.
$$
Therefore,
	\begin{align}\label{Equi4}
	\sum_{i=1}^{\mathrm{I}}   \Delta x_i \mathrm{\Phi}(c_i^{n}) \le A^{n} \sum_{i=1}^{\mathrm{I}}   \Delta x_i \mathrm{\Phi}(c_i^{in})+  B \frac{A^{n} -1}{A-1}.
	\end{align}
	Thanks to Jensen's inequality and having (\ref{convex}), we obtain
		\begin{align}\label{Equi5}
	\int_0^{\mathbb{R}}  \mathrm{\Phi}(c^{h}(t, x))\,dx \le & A^{n} \sum_{i=1}^{\mathrm{I}(h)}   \Delta x_i \mathrm{\Phi} \bigg( \frac{1}{\Delta x_i } \int_{\Lambda_i^h}c^{in}(x) dx \bigg)+  B \frac{A^{n} -1}{A-1}\nonumber\\
	\le & A^{n}\mathcal{I}+  B \frac{A^{n} -1}{A-1} < \infty, \quad \text{for all}\ \ \ t\in [0,T].
	\end{align}
	The computations for Case-(2), Case-(3) and Case-(4) is equavilent to the Case-(1). Only just, we got the different values of A and B, which are the following
	$$  A= \frac{1}{(1-  \lambda \Delta t M_{1}^{in}\|c^{in}\|_{L^1}\,e^{2\lambda R \|b\|_{L^{\infty}} M_{1}^{in} T})},  \,\, 
	B = \frac{\lambda \Delta t R \Pi_{\theta}(\Phi)M_{1}^{in}{\|b\|}^{\theta}_{\infty} \|c^{in}\|_{L^1}\,e^{2\lambda R \|b\|_{L^{\infty}} M_{1}^{in} T}}{(1-\lambda \Delta t M_{1}^{in}\|c^{in}\|_{L^1}\,e^{2\lambda R \|b\|_{L^{\infty}} M_{1}^{in} T})}.
	$$

	 Thus, the sequence $(c^h)$ is said to be weakly compact in $L^1$ by applying the Dunford-Pettis theorem. At the same moment, it is equally bounded with regard to $h$ and $t$, and condition (\ref{36}) is achieved, ensuring the existence of a subsequence of $(c^h)$ that converges weakly to $c \in L^1((0, T)\times (0,R))$ as $h \rightarrow 0$.
	\end{proof}	
	
	The moment has arrived to demonstrate the weak convergence of the sequence $c^n_{i}$, which is formed by a succession of step functions $c^h$. To do this, various point approximations are utilized, which are as seen below. \\
	Midpoint approximation:	
	\begin{align*}
		X^h:x\in (0,R)\rightarrow
		X^h(x)=\sum_{i=1}^{\mathrm{I}}x_i\chi_{\Lambda_i^h}(x).
		\end{align*}
		 Right endpoint approximation:
		\begin{align*}
		\Xi^h:x\in (0,R)\rightarrow
		\Xi^h(x)=\sum_{i=1}^{\mathrm{I}}x_{i+1/2}\chi_{\Lambda_i^h}(x).
		\end{align*}
		Left endpoint approximation:
		\begin{align*}
		\xi^h:x\in (0,R)\rightarrow
		\xi^h(x)=\sum_{i=1}^{\mathrm{I}}x_{i-1/2}\chi_{\Lambda_i^h}(x).
		\end{align*}
		The following lemma is a valuable tool for the convergence.	
\begin{lem}{\label{Wconverge}} [\cite{laurenccot2002continuous}, Lemma A.2]
Let $\Pi$ be an open subset of $\mathbb{R}^m$ and let there exists a constant $l>0$ and two sequences $(z^1_n)_{n\in \mathbb{N}}$ and
$(z^2_n)_{n\in \mathbb{N}}$ such that $(z^1_n)\in L^1(\Pi), z^1\in L^1(\Pi)$ and $$z^1_n\rightharpoonup z^1,\ \ \ \text{weakly in}\
\,L^1(\Pi)\ \text{as}\ n\rightarrow \infty,$$ $(z^2_n)\in
L^\infty(\Pi), z^2 \in L^\infty(\Pi),$ and for all $n\in
\mathbb{N}, |z^2_n|\leq l$ with $$z^2_n\rightarrow z^2,\ \ \text{almost
everywhere (a.e.) in}\ \ \Pi  \ \text{as}\ \ n\rightarrow
\infty.$$ Then
$$\lim_{n\rightarrow \infty}\|z^1_n(z^2_n-z^2)\|_{L^1(\Pi)}=0$$
and $$z^1_n\, z^2_n\rightharpoonup z^1\, z^2,\ \ \ \text{weakly in}\
\,L^1(\Pi)\ \text{as}\ n\rightarrow \infty.$$
\end{lem}\enter		
We have now gathered all  the evidences needed to support Theorem \ref{maintheorem}. To demonstrate this, take a test function $\varphi\in C^1([0,T]\times ]0,R])$ with compact support with respect to $t$ in $[0,t_{N-1}]$ for small $t$. Establish the finite volume for time variable and left endpoint  approximation for space variable of $\varphi$ on $\tau_n\times \Lambda_i^h$ by
$$\varphi_i^n=\frac{1}{\Delta t}\int_{t_n}^{t_{n+1}}\varphi(t,x_{i-1/2})dt.$$  

Multiplying (\ref{fully}) by $\varphi_i^n$ and summing over $n\in \{0,...,N-1\}$ as well as $i\in \{1,...,\mathrm{I}\}$ yield
	\begin{align}\label{convergence1}
	\sum_{n=0}^{N-1}\sum_{i=1}^{\mathrm{I}}\Delta x_{i} (c_{i}^{n+1}-c_{i}^{n})\varphi_i^n = &{\Delta t}\sum_{n=0}^{N-1}\sum_{i=1}^{\mathrm{I}}\sum_{l=1}^{\mathrm{I}}\sum_{j=i}^{\mathrm{I}}K_{j,l}c_{j}^{n}c_{l}^{n}\Delta x_{j} \Delta x_{l}\varphi_i^n \int_{x_{i-1/2}}^{p_{j}^{i}}b(x,x_{j},x_{l})dx\nonumber \\
	&-{\Delta t}\sum_{n=0}^{N-1}\sum_{i=1}^{\mathrm{I}}\sum_{j=1}^{\mathrm{I}}K_{i,j}c_{i}^{n}c_{j}^{n}\Delta x_{i}\Delta x_{j}\varphi_i^n. 		
					\end{align}
When the summation for $n$ is separated, the left-hand side (LHS) resembles like this

	\begin{align*}
	\sum_{n=0}^{N-1}\sum_{i=1}^{\mathrm{I}}\Delta x_{i} (c_{i}^{n+1}-c_{i}^{n})\varphi_i^n = 	\sum_{n=0}^{N-1}\sum_{i=1}^{\mathrm{I}}\Delta	x_i  c_i^{n+1}(\varphi_i^{n+1}-\varphi_i^n)+ \sum_{i=1}^{\mathrm{I}}\Delta
						x_i  c_i^{in}\varphi_i^0.
	\end{align*}
Furthermore, considering the latter equation in terms of the function $c^h$ produces	
	  	
	\begin{align*}
	\sum_{n=0}^{N-1}\sum_{i=1}^{\mathrm{I}}\Delta x_{i} (c_{i}^{n+1}-c_{i}^{n})\varphi_i^n= &\sum_{n=0}^{N-1}\sum_{i=1}^{\mathrm{I}}\int_{\tau_{n+1}}\int_{\Lambda_i^h}c^h(t,x)\frac{\varphi(t,\xi^h(x))-\varphi(t-\Delta t,\xi^h(x))}{\Delta t}dx\,dt \\&+\sum_{i=1}^{\mathrm{I}}\int_{\Lambda_i^h}c^h(0,x)\frac{1}{\Delta t}\int_0^{\Delta t}\varphi(t,\xi^h(x))dt\,dx\\
= &\int_{\Delta t}^T \int_{0}^{R} c^h(t,x)\frac{\varphi(t,\xi^h(x))-\varphi(t-\Delta t,\xi^h(x))}{\Delta t}dx\,dt\\
	& + \int_{0}^R c^h(0,x)\frac{1}{\Delta t}\int_0^{\Delta t}\varphi(t,\xi^h(x))dt\,dx.
	\end{align*}	  
	Since, $\varphi\in C^1([0,T]\times ]0,R])$ posseses compact support and having bounded derivative, $c^h(0,x)\rightarrow c^{in}$ in $L^1(0,R)$ will provide the following result with the help of Lemma \ref{Wconverge}					
	\begin{align}\label{finaltime1}
		\int_{0}^R c^h(0,x)\frac{1}{\Delta t}\int_0^{\Delta t}\varphi(t,\xi^h(x))dtdx\rightarrow \int_{0}^R c^{in}(x)\varphi(0,x)dx
	\end{align}
as max$\{h,\Delta t\}$ goes to $0$.
Now, applying Taylor series expansion of $\varphi$, Lemma \ref{Wconverge} and Proposition \ref{equiintegrability} ensure that for max$\{h,\Delta t\} \rightarrow 0$ 
\begin{align*}
\int_0^T\int_0^R c^h(t,x)\frac{\varphi(t,\xi^h(x))-\varphi(t-\Delta
t,\xi^h(x))}{\Delta t}dx\,dt \rightarrow \int_0^T\int_0^R c(t,x)&\frac{\partial \varphi}{\partial t}(t,x)dx\,dt.
\end{align*}
	Hence, we obtain
	\begin{align}\label{finaltime2}
	&\int_{\Delta t}^T \int_0^R
	\underbrace{c^h(t,x)\frac{\varphi(t,\xi^h(x))-\varphi(t-\Delta t,\xi^h(x))}{\Delta t}}_{c(\varphi)} dx\,dt \nonumber \\ 
	&= \int_0^T \int_0^R c(\varphi)\, dx\,dt -\int_0^{\Delta t} \int_0^R c(\varphi)\, dx\,dt  \rightarrow \int_0^T \int_0^R  c(t,x)\frac{\partial \varphi}{\partial t}(t,x)dx\,dt
	\end{align}
	as max$\{h,\Delta t\}\rightarrow 0$. \\ 
Now, the first term in the RHS of Eq.(\ref{convergence1})	is taken for observing the computation
\begin{align}\label{convergence2}
{\Delta t}\sum_{n=0}^{N-1}\sum_{i=1}^{\mathrm{I}}\sum_{l=1}^{\mathrm{I}}\sum_{j=i}^{\mathrm{I}}K_{j,l}c_{j}^{n}c_{l}^{n}\Delta x_{j} \Delta x_{l}\varphi_i^n \int_{x_{i-1/2}}^{p_{j}^{i}}b(x,x_{j},x_{l})dx \nonumber \\
={\Delta t}\sum_{n=0}^{N-1}\sum_{i=1}^{\mathrm{I}}\sum_{l=1}^{\mathrm{I}}K_{i,l}c_{i}^{n}c_{l}^{n}\Delta x_{i} \Delta x_{l}\varphi_i^n& \int_{x_{i-1/2}}^{x_{i}}b(x,x_{i},x_{l})dx \nonumber \\
 + \Delta t\sum_{n=0}^{N-1}\sum_{i=1}^{\mathrm{I}}\sum_{l=1}^{\mathrm{I}}\sum_{j=i+1}^{\mathrm{I}} K_{j,l}c_{j}^{n}c_{l}^{n}\Delta x_{j} &\Delta x_{l}\varphi_i^n \int_{x_{i-1/2}}^{x_{i+1/2}}b(x,x_{j},x_{l})dx.
\end{align}

The first term of the  Eq.(\ref{convergence2}) simplifies to
\begin{align}\label{convergence3}
\Delta t\sum_{n=0}^{N-1}\sum_{i=1}^{\mathrm{I}}\sum_{l=1}^{\mathrm{I}}K_{i,l}c_{i}^{n}c_{l}^{n}\Delta x_{i} \Delta x_{l}\varphi_i^n \int_{x_{i-1/2}}^{x_{i}}b(x,x_{i},x_{l})dx& \nonumber \\
=\sum_{n=0}^{N-1}\sum_{i=1}^{\mathrm{I}}\sum_{l=1}^{\mathrm{I}}\int_{\tau_{n}}\int_{\Lambda_{i}^{h}}\int_{\Lambda_{l}^{h}}K^h(x,z)c^h(t,x)c^h(t,z)\varphi(t,\xi^h(x))&\int_{\xi^{h}(x)}^{X^{h}(x)}b(r,X^{h}(x),X^{h}(z))dr\,dz\,dx\,dt\nonumber \\
= \int_{0}^{T}\int_{0}^{R}\int_{0}^{R}K^h(x,z)c^h(t,x)c^h(t,z)\varphi(t,\xi^h(x))\int_{\xi^{h}(x)}^{X^{h}(x)}&b(r,X^{h}(x),X^{h}(z))dr\,dz\,dx\,dt.
\end{align}
Next, the second term of Eq.(\ref{convergence2}) leads to
\begin{align}\label{convergence4}
\Delta t\sum_{n=0}^{N-1}\sum_{i=1}^{\mathrm{I}}\sum_{l=1}^{\mathrm{I}}\sum_{j=i+1}^{\mathrm{I}} K_{j,l}c_{j}^{n}c_{l}^{n}\Delta x_{j} \Delta x_{l}\varphi_i^n \int_{x_{i-1/2}}^{x_{i+1/2}}b(x,x_{j},x_{l})dx \nonumber \\
=\sum_{n=0}^{N-1}\sum_{i=1}^{\mathrm{I}}\sum_{l=1}^{\mathrm{I}}\sum_{j=i+1}^{\mathrm{I}}\int_{\tau_{n}}\int_{\Lambda_{i}^{h}}\int_{\Lambda_{l}^{h}}\int_{\Lambda_{j}^{h}}K^h(y,z)c^h(t,y)c^h(t,z)\varphi(t,\xi^h(x)) \nonumber \\ \frac{1}{\Delta x_{i}}\int_{\Lambda_{i}^{h}}b(r,X^{h}(y),X^{h}(z))dr\,dy\,dz\,dx\,dt \nonumber \\
= \int_{0}^{T}\int_{0}^{R}\int_{0}^{R}\int_{\Xi^{h}(x)}^{R} K^h(y,z)c^h(t,y)c^h(t,z)\varphi(t,\xi^h(x))b(X^{h}(x),X^{h}(y),X^{h}(z))dy\,dz\,dx\,dt.
\end{align}
Eqs.(\ref{convergence2})-(\ref{convergence4}), Lemma \ref{Wconverge} and Proposition \ref{equiintegrability} imply that as max$\{h,\Delta t\} \rightarrow 0$ 
\begin{align}\label{convergence5}
{\Delta t}\sum_{n=0}^{N-1}\sum_{i=1}^{\mathrm{I}}\sum_{l=1}^{\mathrm{I}}\sum_{j=i}^{\mathrm{I}}K_{j,l}c_{j}^{n}c_{l}^{n}\Delta x_{j} \Delta x_{l}\varphi_i^n \int_{x_{i-1/2}}^{p_{j}^{i}}b(x,x_{j},x_{l})dx \nonumber \\
\rightarrow \int_{0}^{T}\int_{0}^{R}\int_{0}^{R}\int_{x}^{R} K(y,z)c(t,y)c(t,z)\varphi(t,x)&b(x,y,z)dy\,dz\,dx\,dt.
\end{align}
Taking the second term on the RHS of Eq.(\ref{convergence1}) employs
	\begin{align}\label{convergence6}
	{\Delta t}\sum_{n=0}^{N-1}\sum_{i=1}^{\mathrm{I}}\sum_{j=1}^{\mathrm{I}}K_{i,j}c_{i}^{n}c_{j}^{n}\Delta x_{i}\Delta x_{j}\varphi_i^n	\nonumber\\=
	\sum_{n=0}^{N-1}\sum_{i=1}^{\mathrm{I}}&\sum_{j=1}^{\mathrm{I}}\int_{\tau_{n}}\int_{\Lambda_i^h}\int_{\Lambda_j^h}K^h(x,y)c^h(t,x)c^h(t,y)\varphi(t,\xi^h(x))dy\,dx\,dt \nonumber	\\
	\rightarrow \int_{0}^{T}&\int_{0}^{R}\int_{0}^{R}K(x,y)c(t,x)c(t,y)\varphi(t,x)dy\,dx\,dt
	\end{align}
	as  max$\{h,\Delta t\} \rightarrow 0$.
	Eqs.(\ref{convergence1})-(\ref{convergence6}) deliver the desired results for the weak convergence as presented in Eq.(\ref{convergence0}).

\section{Error Simulation }\label{Error}
In this section, the error estimation is explored for CBE, which is based on the idea of \cite{bourgade2008convergence}. Taking the uniform mesh is crucial for estimating the error component, i.e., $\Delta x_i=h$\,\,$\forall i\in \{1,2,..., \mathrm{I}\}$. The error estimate is achieved by providing a  estimations on the difference $c^h-c$, where $c^h$ is constructed using the numerical technique and $c$ represents the exact solution to the problem (\ref{maineq}).
By using the following Theorem, we can determine the error estimate by making some assumptions about the kernels and the initial datum.\\

	\begin{thm}\label{errorth1}
	Let the collisional and breakage kernels satisfy $K\in W^{1,\infty}_{loc}(\mathbb{R}^{+} \times \mathbb{R}^{+})$, $b\in W^{1,\infty}_{loc}(\mathbb{R}^{+} \times \mathbb{R}^{+}\times \mathbb{R}^{+})$ and  initial datum  $c^{in}\in W^{1,\infty}_{loc}(\mathbb{R}^{+})$.	Moreover, consider a  uniform volume mesh and  time step $\Delta t$  that satisfy the condition (\ref{22}). Then, the following error estimates
	\begin{align}\label{errorbound}
	\|c^h-c\|_{L^{\infty}(0,T;L^{1}(0, R))} \leq H(T, R)(h+\Delta t)
	\end{align}
	holds,	where $c$ is the weak solution to (\ref{maineq}). 
	\end{thm}
	
Before proving the Theorem, consider the following proposition, which provides an estimate on the approximate solution $c^h$ and the exact solution $c$ given certain additional assumptions. These estimates are important in the analysis of  the error.
	\begin{propos}\label{bound2}
Assume that kinetic parameters $K \in L^{\infty}_{loc}(\mathbb{R}^{+} \times \mathbb{R}^{+})$, $b\in L^{\infty}_{loc}(\mathbb{R}^{+} \times \mathbb{R}^{+}\times \mathbb{R}^{+})$ and  the condition (\ref{22}) holds for time step $\Delta t$. Also, let the initial  datum $c^{in}$  restricted in $ L^{\infty}_{loc}$. Then, solution $c^h$ and $c$ to (\ref{maineq}) are essentially bounded in $(0,T)\times (0, R)$ as
$$ \|c^h\|_{L^{\infty}((0,T)\times (0, R))}\leq H(T, R), \hspace{0.4cm} \|c\|_{L^{\infty}((0,T)\times (0, R))}\leq H(T, R). $$
Furthermore, if the kernels $K \in    W^{1,\infty}_{loc}(\mathbb{R}^{+} \times \mathbb{R}^{+})$, $b\in W^{1,\infty}_{loc}(\mathbb{R}^{+} \times \mathbb{R}^{+}\times \mathbb{R}^{+})$ and $c^{in} \in W^{1,\infty}_{loc}(\mathbb{R}^{+})$. Then there exists a positive constant $H(T, R)$ such that
\begin{align}\label{bound1}
\|c\|_{W^{1,\infty}(0, R)} \leq H(T, R).
\end{align}
\end{propos}
\begin{proof}
The purpose is to connect the continuous Eq.(\ref{maineq}) to the bounded solution $c$. In consequence  integrating Eq.(\ref{trunceq}) with respect to the time variable provides the following result
	\begin{align*}
c(t,x) \leq&c^{in}(x)+ \int_{0}^{t}\int_0^R\int_{x}^{R} K(y,z)b(x,y,z)c(s,y)c(s,z)dy\,dz\,ds \\ 
& \leq c^{in}(x)+\|K\|_{\infty}\|b\|_{\infty}{\|c\|}^{2}_{\infty,1}t,
\end{align*}	
	where $\|c\|_{\infty,1}$ represents the norm of $c$ in $L^{\infty}(0,T; L^{1}(\,0, R)\,)$. 
\begin{align*}
\|c\|_{L^{\infty}((0,T) \times (\,0, R)\,)} \leq H(T, R).
\end{align*}
Now, let us go to the culmination of an analysis of (\ref{bound1}). First, integrate Eq.(\ref{trunceq}) for the time variable $t$, and then differentiate it with respect to the volume variable $x$ yields 
\begin{align*}
\frac{\partial c(t,x)}{\partial x} \leq & \frac{\partial c^{in}(x)}{\partial x}+ \frac{\partial}{\partial x} \int_{0}^{t}\int_0^R\int_{x}^{R} K(y,z)b(x,y,z)c(s,y)c(s,z)dy\,dz\,ds \\
& -\frac{\partial}{\partial x} \int_{0}^{t}\int_{0}^{R}K(x,y)c(t,x)c(t,y)\,dy,
\end{align*}
use of the maximum value across the domain of $x$ and simplification of compuatation yield the following condition 
\begin{align*}
\left\|{\frac{\partial c(t)}{\partial x}}\right\|_{\infty} \leq & \left\|{\frac{\partial c^{in}(x)}{\partial x}}\right\|_{\infty}+ [\|Kb\|_{\infty}\|c\|_{\infty,1}\|c\|_{\infty}+ \|K\|_{\infty}\|b\|_{W^{1,\infty}}{\|c\|}^{2}_{\infty,1}\\
& + \|K\|_{W^{1,\infty}}\|c\|_{\infty,1}\|c\|_{\infty}
]t+ \|K\|_{\infty}\|c\|_{\infty,1}\int_{0}^{t}\left \|\frac{\partial c}{\partial x}\right\|_{\infty}\,ds,
\end{align*}
it has been written in a more coherent way
\begin{align*}
\left\|{\frac{\partial c(t)}{\partial x}}\right\|_{\infty}\leq \Upsilon (t)+\upsilon \int_{0}^{t}\left \|\frac{\partial c}{\partial x}\right\|_{\infty}\,ds,
\end{align*}
where 
$$\Upsilon (t)=\left\|{\frac{\partial c^{in}(x)}{\partial x}}\right\|_{\infty}+ [\|Kb\|_{\infty}\|c\|_{\infty,1}\|c\|_{\infty}+ \|K\|_{\infty}\|b\|_{W^{1,\infty}}{\|c\|}^{2}_{\infty,1}
 + \|K\|_{W^{1,\infty}}\|c\|_{\infty,1}\|c\|_{\infty}
]t,$$
$$\upsilon =\|K\|_{\infty}\|c\|_{\infty,1}.$$
Beyond that, the use of Gronwall's lemma and integration by parts establish the proof as follows
\begin{align*}
\left\|{\frac{\partial c(t)}{\partial x}}\right\|_{\infty} & \leq  \Upsilon (t)+ \int_{0}^{t}\Upsilon (s) \upsilon e^{\int_{s}^{t}\upsilon\,dr}\,ds\\	
& \leq \Upsilon (0) e^{\upsilon t} + (\|Kb\|_{\infty}\|c\|_{\infty,1}\|c\|_{\infty}+ \|K\|_{\infty}\|b\|_{W^{1,\infty}}{\|c\|}^{2}_{\infty,1}\\
&+ \|K\|_{W^{1,\infty}}\|c\|_{\infty,1}\|c\|_{\infty})[(e^{\upsilon t}-1)].
\end{align*}
Therefore
\begin{align*}
\left\|{\frac{\partial c}{\partial x}}\right\|_{L^{\infty}((0,T) \times (\,0, R)\,)} \leq H(T, R).
\end{align*}
It concludes the result (\ref{bound1}).
\end{proof}
The discrete collisional birth-death term given as in (\ref{fully}) expressed like 
\begin{align}\label{errord}
B_{C}(i)-D_{C}(i)=\frac{1}{\Delta x_i}\sum_{l=1}^{\mathrm{I}}\sum_{j=i}^{\mathrm{I}}K_{j,l}c^{n}_{j}c^{n}_{l}\Delta x_{j}\Delta x_{l}\int_{x_{i-1/2}}^{p_{j}^{i}}b(x,x_{j},x_{l})\,dx-\sum_{j=1}^{\mathrm{I}}K_{i,j}c^{n}_{i}c^{n}_{j}\Delta x_{j}
\end{align}	
The subsequent lemma offers a simplified version of the preceding discrete terms.
\begin{lem}
Consider the initial condition $c^{in}$ $\in W^{1,\infty}_{loc}$ and uniform mesh, $\Delta x_{i}=h$ $\forall i$. Also assuming that $K$ and $b$ follow the conditions   $K,b\in W^{1,\infty}_{loc}.$ Let $(s,x)\in \tau_{n}\times \Lambda_{i}^{h}$, where $n\in \{0,1,..., N-1\}\, , i\in\{1,2,...,\mathrm{I}\}$.Then
\begin{align}\label{convert1}
B_{C}(i)-D_{C}(i)=&\int_{0}^{R}\int_{{\Xi}^{h}(x)}^{R}K^h(y,z)b^{h}(x,y,z)c^{h}(s,y)c^{h}(s,z)\,dydz \nonumber \\ &-\int_{0}^{R}K^h(x,y)c^{h}(s,x)c^{h}(s,y)\,dy+\varepsilon(h),
\end{align}
In the strong $L^1$ topology, $\varepsilon(h)$  defines the first order term with regard to $h$: 
\begin{align}
\|\varepsilon(h)\|_{L^1}\leq \frac{\|Kb\|_{L^{\infty}}}{2}{\|c^{in}\|}^{2}_{L^1}\,e^{2\gamma R \|b\|_{L^{\infty}} M_{1}^{in} T} h.
\end{align}
\end{lem}
\begin{proof}
Initiate with a discrete birth term of Eq.(\ref{errord}) and convert it to a continuous form with a uniform mesh and $x \in \Lambda_{i}^{h}$,
\begin{align}\label{error1}
&\frac{1}{\Delta x_i}\sum_{l=1}^{\mathrm{I}}\sum_{j=i}^{\mathrm{I}}K_{j,l}c^{n}_{j}c^{n}_{l}\Delta x_{j}\Delta x_{l}\int_{x_{i-1/2}}^{p_{j}^{i}}b(x,x_{j},x_{l})\,dx \nonumber \\
&= \sum_{l=1}^{\mathrm{I}}\sum_{j=i+1}^{\mathrm{I}}K_{j,l}c^{n}_{j}c^{n}_{l}\Delta x_{j}\Delta x_{l}\frac{1}{\Delta x_i}\int_{x_{i-1/2}}^{x_{i+1/2}}b(x,x_{j},x_{l})\,dx
+ \sum_{l=1}^{\mathrm{I}}K_{i,l}c^{n}_{i}c^{n}_{l}\Delta x_{l}\int_{x_{i-1/2}}^{x_{i}}b(x,x_{i},x_{l})\,dx \nonumber \\
&= \int_{0}^{R}\int_{{\Xi}^{h}(x)}^{R}K^h(y,z)b^{h}(x,y,z)c^{h}(s,y)c^{h}(s,z)\,dydz+\varepsilon(h),
\end{align}
where $\varepsilon(h)= \sum_{l=1}^{\mathrm{I}}K_{i,l}c^{n}_{i}c^{n}_{l}\Delta x_{l}\int_{x_{i-1/2}}^{x_{i}}b(x,x_{i},x_{l})\,dx$ is defined. Calculating the $L^{1}$ norm of $\varepsilon(h)$ leads to the following term
\begin{align*}
\|\varepsilon(\mathbb{F},h)\|_{L^{1}}&\leq \|Kb\|_{L^{\infty}}\sum_{i=1}^{\mathrm{I}}c_{i}^{n}\Delta x_{i}\sum_{l=1}^{\mathrm{I}}c_{l}^{n}\Delta x_{l}\int_{x_{i-1/2}}^{x_{i}}\,dx\\\
& \leq \frac{\|Kb\|_{L^{\infty}}}{2}{\big(\sum_{i=1}^{\mathrm{I}}c_{i}^{n}\Delta x_{i}}\big)^{2}h\\
& \leq \frac{\|Kb\|_{L^{\infty}}}{2}{\|c^{in}\|}^{2}_{L^1}\,e^{2\gamma R \|b\|_{L^{\infty}} M_{1}^{in} T} h.
\end{align*}	
Now, taking the  discrete death term of (\ref{errord})
\begin{align}\label{error2}
\sum_{j=1}^{\mathrm{I}}K_{i,j}c^{n}_{i}c^{n}_{j}\Delta x_{j}=\int_{0}^{R}K^h(x,y)c^{h}(s,x)c^{h}(s,y)\,dy.
\end{align}
Using the formula (\ref{convert1}), Eq.(\ref{trunceq}) and Eq.(\ref{fully}), lead to error formulation for $t\in \tau_n$ as
\begin{align}\label{errorfull}
\int_{0}^{R}|c^h(t,x)-c(t,x)|dx \leq  \int_{0}^{R}|c^h(0,x)-c(0,x)|dx + \sum_{\beta=1}^{3}(CB)_{\beta}(h)\nonumber \\ 
+\int_{0}^{R}|\epsilon(t,n)|\,dx+ \|\varepsilon(h)\|_{L^{1}} t,
\end{align}
where error terms are expressed by $(CB)_{\beta}(h)$ for  $\beta=1,2,3$ 
\begin{align*}
(CB)_{1}(h)= \int_{0}^{t}\int_{0}^{R}\int_{0}^{R}\int_{\Xi^{h}(x)}^{R}|K^h(y,z)b^{h}(x,y,z)c^{h}(s,y)c^{h}(s,z)\\-K(y,z)b(x,y,z)c(s,y)c(s,z)|\,dy\,dz\,dx\,ds,
\end{align*}
\begin{align*}
(CB)_{2}(h)= \int_{0}^{t}\int_{0}^{R}\int_{0}^{R}\int_{x}^{\Xi^{h}(x)}K(y,z)b(x,y,z)c(s,y)c(s,z)\,dy\,dz\,dx\,ds,		
\end{align*}
and
\begin{align*}
(CB)_{3}(h)= \int_{0}^{t}\int_{0}^{R}\int_{0}^{R}|K^h(x,y)c^{h}(s,x)c^{h}(s,y)-K(x,y)c(s,x)c(s,y)|\,dy\,dx\,ds.
\end{align*}
Considering $|t-t_n|\leq \Delta t$,  the time discretization provides the following expression
\begin{align*}
\int_{0}^{R}|\epsilon(t,n)|\,dx\leq & \int_{t_n}^{t}\int_{0}^{R}\int_{0}^{R}\int_{\Xi^{h}(x)}^{R}K^h(y,z)b^{h}(x,y,z)c^{h}(s,y)c^{h}(s,z)\,dy\,dz\,dx\,ds \nonumber \\
& + \int_{t_n}^{t}\int_{0}^{R}\int_{0}^{R}K^h(x,y)c^{h}(s,x)c^{h}(s,y)\,dy\,dx\,ds \nonumber \\
&+ \int_{t_n}^{t}\int_{0}^{R} \varepsilon(h)\,dx\,ds.
\end{align*}
Given	$K, b\in W_{loc}^{1,\infty}$, we have $x,y\in (0,R]$ for 
	$$	|K^{h}(x,y)-K(x,y)|\leq \|K\|_{W^{1,\infty}} h. $$
As a result, it produces an estimate of $(CB)_{1}(h)$ using the $L^1$ bound on $c^h$ and $c$. To begin, divide the expression into four segments
\begin{align*}
(CB)_{1}(h)\leq & 	\int_{0}^{t}\int_{0}^{R}\int_{0}^{R}\int_{0}^{R}|K^h(y,z)-K(y,z)|b(x,y,z)c(s,y)c(s,z)\,dy\,dz\,dx\,ds\\
& + \int_{0}^{t}\int_{0}^{R}\int_{0}^{R}\int_{0}^{R}K^h(y,z)|b^{h}(x,y,z)-b(x,y,z)|c(s,y)c(s,z)\,dy\,dz\,dx\,ds\\
	& +  \int_{0}^{t}\int_{0}^{R}\int_{0}^{R}\int_{0}^{R} K^h(y,z)b^{h}(x,y,z)|c^{h}(s,y)-c(s,y)|c(s,z)\,dy\,dz\,dx\,ds  \\
	& + \int_{0}^{t}\int_{0}^{R}\int_{0}^{R}\int_{0}^{R}K^h(y,z)b^{h}(x,y,z)c^{h}(s,y)|c^{h}(s,z)-c(s,z)|\,dy\,dz\,dx\,ds.
	\end{align*}
By simplifying and employing Proposition (\ref{bound2}), the above may be transformed to
	\begin{align}\label{error3}
	(CB)_{1}(h) &\leq (\|K\|_{W^{1,\infty}} \|b\|_{\infty}+\|K\|_{\infty} \|b\|_{W^{1,\infty}})t  R^{3}  {\|c\|}^{2}_{\infty}   h \nonumber \\
	& + R^{2}\|K\|_{\infty}\|b\|_{\infty}(\|c\|_{\infty}+ \|c^h\|_{\infty})\int_{0}^{t}\|c^{h}(s)-c(s)\|_{L^1}\,ds,
	\end{align}
	similar estimation for $(CB)_{3}(h)$
	\begin{align}\label{error4}
	(CB)_{3}(h) \leq \|K\|_{W^{1,\infty}}t  R^{2}  {\|c\|}^{2}_{\infty}   h 
		 + R\|K\|_{\infty}(\|c\|_{\infty}+ \|c^h\|_{\infty})\int_{0}^{t}\|c^{h}(s)-c(s)\|_{L^1}\,ds.
	\end{align}
	Moving on to the remaining terms, $(CB)_{2}(h)$ and $\int_{0}^{R}|\epsilon(t,n)|\,dx$, it is clear that
	\begin{align}\label{error5}
	 (CB)_{2}(h)\leq \frac{t R^2}{2}\|Kb\|_{\infty}{\|c\|}^{2}_{\infty} h,
	\end{align}
	and 
	\begin{align}\label{error6}
	\int_{0}^{R}|\epsilon(t,n)|\,dx \leq (\|Kb\|_{\infty}{\|c^h\|}^{2}_{\infty}R^3+\|K\|_{\infty}{\|c^h\|}^{2}_{\infty}R^2+\|\epsilon(h)\|_{L^{1}})\Delta t.
	\end{align}
	Furthermore, substituting all of the estimations  (\ref{error3})-(\ref{error6}) in (\ref{errorfull}) and applying the Gronwall's lemma to conclude the result in (\ref{errorbound}).
\end{proof}
\section{Numerical Testing } \label{testing}
In this part of the article, the discussion over experimental error and experimental order of convergence (EOC) has been concluded for three combinations of collision kernel and breakage distribution function. As we are aware that the kernels must exist in  $ W^{1,\infty}_{loc}$ space for error estimation with uniform meshes. In two cases, no theoretical results are available in the literature. To test this problem's theoretical error estimation, we have elected three cases with exponential initial condition (IC) $c(0,x)=\exp(-x)$. To validate the result, the following collision  kernel ($K$), breakage distribution function ($b$) and IC  combinations are used:\\

\textbf{Test case 1:}  $ K(y,z)=1,  b(x,y,z)=\frac{(\aleph+2)x^{\aleph}}{(y)^{\aleph+1}}, -1<\aleph \leq 0$ with $\aleph=0$.\\
\textbf{Test case 2:}  $K(y,z)=y+z,  b(x,y,z)= \frac{(\aleph+2)x^{\aleph}}{(y)^{\aleph+1}}, -1< \aleph \leq 0$ with $\aleph=0$.\\
\textbf{Test case 3:} $ K(y,z)=y+z,   b(x,y,z)=\delta(x-0.4y)+\delta(x-0.6y)$.\\

The experimental domain of volume is [1e-3, 10] discretized into 30,60,120,240, and 480 cells and computations run from time 0 to 0.2. In order to observe the EOC of the FVS in each cell of the computational domain, the following relation is used to estimate result:
	\begin{align}
	EOC=\ln \left( \frac{\|N_{\mathrm{I}}-N_{2\mathrm{I}}\|}{\|N_{2\mathrm{I}}-N_{4\mathrm{I}}\|} \right)/\ln(2).
	\end{align}
Here, $N_{\mathrm{I}}$ denotes the total number of particles generated by the FVS (\ref{fully})  with a mesh of $\mathrm{I}$ number of cells.

	\begin{table}[!htb]
	 \begin{minipage}{.5\linewidth}
	  \centering
	    \begin{tabular}{ |p{1cm}|p{2cm}|p{1cm}|}
	    \hline
	      Cells      & Error & EOC\\
	       \hline
	       30    & - & -  \\
	          \hline
	      60 & 0.4377$\times 10^{-4}$ & - \\
	           \hline
	       120  &   0.2047$\times 10^{-4}$   &1.0963   \\
	          \hline
	        240  & 0.0989$\times 10^{-4}$ & 1.0501  \\
	          \hline
	      480  & 0.0485$\times 10^{-4}$ & 1.0265   \\
	 \hline
	           \end{tabular}
	            \caption{ Test case 1}
	           \end{minipage}
	           \begin{minipage}{.5\linewidth}
	           \centering
	            \begin{tabular}{ |p{1cm}|p{2cm}|p{1cm}|}
	            \hline
	            Cells      & Error & EOC\\
	            \hline
	              30    & - &-   \\
	                 \hline
	              60& 0.4879$\times 10^{-4}$  &- \\
	             \hline
	              120   & 0.2196$\times 10^{-4}$    & 1.1515 \\
	                 \hline
	               240  & 0.1032$\times 10^{-4}$ & 1.0901  \\
	               \hline
	               480  & 0.0498$\times 10^{-4}$ & 1.0497    \\
	            \hline
	            \end{tabular}
	             \caption{ Test case 2}
	              \end{minipage}
	          
	           \end{table} 
	            
	           	\begin{table}[!htb]
	           	
	           	  \centering
	           	    \begin{tabular}{ |p{1cm}|p{2cm}|p{1cm}|}
	           	    \hline
	           	      Cells      & Error & EOC\\
	           	       \hline
	           	       30    & - & -  \\
	           	          \hline
	           	      60 & 0.4309$\times 10^{-4}$ & - \\
	           	           \hline
	           	       120  &   0.2023$\times 10^{-4}$   &1.0905   \\
	           	          \hline
	           	        240  & 0.0981$\times 10^{-4}$ & 1.0452  \\
	           	          \hline
	           	      480  & 0.0483$\times 10^{-4}$ & 1.0226   \\
	           	 \hline
	           	           \end{tabular}
	           	            \caption{ Test case 3}
	           	           \end{table}
Tables 1, 2 and 3 represent the error and EOC for uniform mesh. In addition, The numerical errors are measured using  30, 60, 120, 240, and 480  cells, and the scheme provides the error in decreasing mode. The tables depict that the FVS yields first-order convergence, as predicted by theoretical results in section \ref{Error}.
\section{Conclusion}\label{conclusion}
This article proposes the finite volume scheme  for the collisional breakage equation for the non-uniform mesh. It yields a non-conservative scheme, for which a weak convergence analysis has been executed with unbounded collision and breakage distribution kernels. Where numerical truncated solution convergences to a weak solution of the problem. It is accomplished in the presence of the Weak $L^1$ compactness method based on Dunford-Pettis and La Vall$\acute{e}$e Poussin theorems.  In addition, explicit error estimation of the method is also explored for the locally bounded kernels. It has been demonstrated that the FVS
is first-order accurate for uniform meshes. Moreover, We also compared experimental result to theoretical result for various combinations of collision kernel and breakage distribution function.

\bibliography{colref}
\bibliographystyle{ieeetr}
				\end{document}